\documentclass[12pt]{article}  
\pdfoutput=1 	
\usepackage{geometry}                		
\usepackage[parfill]{parskip}    		
\usepackage{graphicx}			
\usepackage[hyperindex=true,pagebackref,colorlinks=true,pdfpagemode=none,urlcolor=blue,linkcolor=blue,citecolor=blue,pdfstartview=FitH]{hyperref}
\usepackage{amsmath,amsfonts}
\usepackage{color}
\usepackage{amsthm}
\usepackage{boolexpr}
\usepackage{amssymb,latexsym,pstricks,pst-plot}
\usepackage[english]{babel}
\usepackage{pgf}
\usepackage{tikz}
\usetikzlibrary{arrows,automata}
\usepackage[UKenglish]{isodate}
\usepackage[T1]{fontenc}
\usepackage{listings}
\usepackage{xcolor}
\usepackage{bm}
\usepackage{mathabx}
\usepackage{enumitem}
\usepackage{mathtools}
\allowdisplaybreaks

\usepackage{amssymb}
\newtheorem{theorem}{Theorem}[section]
\newtheorem{lemma}[theorem]{Lemma}

\newtheorem{corollary}[theorem]{Corollary}

\theoremstyle{definition}

\newtheorem*{remark}{Remark}

\newcommand\smallO{
  \mathchoice
    {{\scriptstyle\mathcal{O}}}
    {{\scriptstyle\mathcal{O}}}
    {{\scriptscriptstyle\mathcal{O}}}
    {\scalebox{.7}{$\scriptscriptstyle\mathcal{O}$}}
  }

\title{Synchronisation of running sums of automatic sequences}
\author{Rob Burns}

\begin{document}
\maketitle
\begin{abstract}
We investigate the running sums of some well-known automatic sequences to determine whether they are synchronised.
\end{abstract}

\section{Introduction}
\label{intro}

Synchronised sequences were introduced by Carpi and Maggi  \cite{Carpi_2001}  in 2001.These sequences, or functions over $\mathbb{N}$, lie between automatic and regular sequences. Roughly speaking, a function $f$ is called synchronised if there is an automaton which takes as input a pair of integers $(n, s)$, written in some number system(s), and accepts the pair if and only if $f(n) = s$. If both $n$ and $s$ are written in base $k$, then the function is called $k$-synchronised. If $n$ is written in base $k$ and $s$ is written in base-$l$, then the function is called $(k, l)$-synchronised. If $n$ and $s$ are written in Fibonacci representation, then the function is called Fibonacci-synchronised.

The running sum of a sequence of positive integers $a = ( a_i )_{i=0}^{\infty}$ is defined to be 
$$
f_a (n) := \sum_{i=0}^{n} a_i .
$$

Unless the sequence $a$ is trivial, $f_a$ is unbounded and therefore not automatic. On the other hand, if $a$ is automatic, then $f_a$ is regular.  

This paper examines some standard automatic sequences to determine whether their running sums are synchronised. In his book \textit{The Logical Approach to Automatic Sequences}\cite{Shallit_2022}, Jeffrey Shallit provides a Fibonacci-synchronising automaton for the running sum of the Fibonacci word (section 10.11) and a $2$-synchronising automaton for the running sum of the ternary Thue-Morse sequence (section 10.15).

\bigskip

\section{Preliminaries}
\label{bg}

The Fibonacci numbers $0, 1, 1, 2, 3, 5, 8, \dots$ are defined by the recursive relations
$$
F_0 = 0, \,\,\, F_1 = 1 \,\,\, \text{ and } \,\,\, F_{n} = F_{n-1} + F_{n-2} \text{ for } \,\,\, n \geq 2.
$$
The golden ratio is defined by $\phi := (1 + \sqrt{5} )/2$. The conjugate of $\phi$ is defined by $\psi := (1 - \sqrt{5} )/2 = -1/\phi$. $\phi$ and $\psi$ are roots of the polynomial $x^2 - x - 1$. The Fibonacci numbers satisfy Binet's formula:
\begin{equation}
\label{binet}
F_n = (\phi^n - \psi^n )/\sqrt{5}.
\end{equation}

We will sometimes use regular expressions to represent integers. The regular expression for an integer depends on the base in which it is being represented. For example, in base $2$, non-zero even numbers are represented by the regular expression
$$
0^* 1 (1 | 0)^* 0.
$$
Note that the regular expression above has been padded with $0$'s at the most significant digit end of its representation. This will generally be the case. We will use exponentiation when the regular expression contains a repeated word. For example, in base $3$, the base $10$ integer $30$ is represented by the expression $(1 0)^2 = 1 0 1 0$. Sequences defined by regular expressions can also be defined by an automaton and vice versa.

There is a connection between an extension of Presburger arithmetic and automatic sequences as described in theorem 6.6.1 of  \cite{Shallit_2022}. As a result, two or more automatic sequences (or synchronised sequences) can be combined using operations from first-order logic to create new automata (or synchronised sequences). If $f$ is a synchronised function and $s$ is an automatic sequence, then we can create another synchronised function $f^{\prime}$ by
$$
f^{\prime}(n)  =
\begin{cases}
    f(n), & \text{if } \, \, n \in s \\
    0, & \text{if } \, \, n \not\in s.
\end{cases}
$$

We will use this fact in section \ref{unsync}.

The automata in this paper will read integers starting with the most significant digit (msd).

We have used Shallit's book \cite{Shallit_2022} as the source for information about automatic sequences and the software package {\tt Walnut}. We used  {\tt Walnut} to handle automata. Hamoon Mousavi, who wrote the program, has provided an introductory article \cite{Mousavi:2016aa}. Further resources and a  list of papers that have used {\tt Walnut} can be found at \href{https://www.cs.uwaterloo.ca/~shallit/papers.html}{Jeffrey Shallit's page}.

The free open-source mathematics software system SageMath \cite{sagemath} was used to perform the matrix algebra and polynomial factorisation.

We will need the following theorem from \cite{Shallit_2021}, which establishes some closure properties of synchronised sequences. Here, $\lfloor . \rfloor$ is the floor function.

\bigskip

\begin{theorem} 
\label{shallit}
Suppose $(a(n))_{n \geq 0}$ and $(b(n))_{n \geq 0}$ are $(k, l)$-synchronised sequences. Then so are the sequences
\begin{enumerate}[label=(\alph*)]
\item  $(a(n)+ b(n))_{n \geq 0}$;  
\item  $(a(n) \dotdiv b(n))_{n \geq 0}$, \text{ where } $x \dotdiv y$ \text{ is the "monus" function, defined by} $\max(x-y, 0)$; 
\item  $( \lfloor a(n) - b(n) \rfloor )_{n \geq 0} $;
\item $ (\lfloor \alpha a(n) \rfloor )_{n \geq 0}$, \text{ where } $\alpha$ \text {is a non-negative rational number}.
\end{enumerate}
\end{theorem}

\bigskip

We will also need the following growth bounds on synchronised sequences. See \cite{Shallit_2022}, theorem 10.6.1.

\bigskip

\begin{theorem}
\label{shallit2}
Let $k, l \geq 2$ be integers. Let $(f(n))_{n \geq 0}$ be a $(k, l)$-synchronised sequence, and define $\beta = (\log l)/(\log k)$. Then
\begin{enumerate}[label=(\alph*)]
\item  $f(n) = \mathcal{O} (n^{\beta})$;
\item  if  $f(n) = \smallO (n^{\beta})$, then $f(n) =  \mathcal{O} (1)$.
\end{enumerate}
\end{theorem}

\bigskip

\begin{remark}
The bounds of theorem \ref{shallit2} also hold for Fibonacci-synchronised sequences (with $\beta = 1$) (See the comment on p. 284 of  \cite{Shallit_2022}).
\end{remark}

\bigskip

\begin{remark}
An example of how theorem \ref{shallit2} can be used to show a sequence is not synchronised is provided by the characteristic sequence of the powers of $2$. The running sum of this sequence is unbounded but has growth rate $\mathcal{O} (\log(n)) = \smallO (n)$ and so cannot be $2$-synchronised.
\end{remark}
\bigskip

\section{The running sum and the index function}
\label{relation}

Let $s$ be an automatic binary sequence. Define  $\bm{\mathrm{sum_s}}$ to be the running sum of $s$. Define  $\bm{\mathrm{ind_s}}$ to be the sequence of indices at which the binary digit $1$ appears in $s$. For example, for the Thue-Morse sequence $\bm{\mbox{T}} = (0, 1, 1, 0, 1, 0, 0, 1, 1, 0, 0, 1, 0. \dots)$, we have
$$
 \bm{\mathrm{sum_T}} = (0, 1, 2, 2, 3, 3, 3, 4, 5, 5, \dots) \,\,\, \text{    and   } \,\,\, \bm{\mathrm{ind_T}} = (1, 2, 4, 7, 8, 11, 13, 14, 16, \dots).
$$

\bigskip

\begin{theorem}
If $s$ is an automatic binary sequence, then the sequence  $\bm{\mathrm{sum_s}}$ is synchronised if and only if the sequence $\bm{\mathrm{ind_s}}$ is synchronised.
\end{theorem}
\begin{proof}
If $B$ is a synchronising automaton for $\bm{\mathrm{sum_s}}$, then $B$ accepts the pair $(k, n)$ if and only if $\bm{\mathrm{sum_s}}$$(k) = n$. Define the automaton $A$ by the logical formula
$$
A(n,k) :=  B(k,n) \land s(k) = 1
$$
Then, $A$ is a synchronising automaton for $\bm{\mathrm{ind_s}}$.
Conversely, suppose $A$ is a synchronising automaton for $\bm{\mathrm{ind_s}}$. So, $A$ accepts the pair $(n, k)$ if and only if the $n$'th $1$ in $s$ appears at index $k$, i.e. $\bm{\mathrm{ind_s}}$$(n) = k$. Define the automaton $B$ by the first order logical expression
$$
B(k,n) :=  \exists r (r \leq k): A (n,r) \land (\forall i: r+1 \leq i \leq k \,\,\, s(i) = 0).
$$
Then, $B$ is a synchronising automaton for  $\bm{\mathrm{sum_s}}$.
\end{proof}

\bigskip

\section{Synchronised sequences}
\label{sync}

We will now examine some automatic sequences for which the running sums are synchronised. Our method is to first guess the synchronising automaton using the Myhill-Nerode theorem and then verify the guess with  {\tt Walnut}.

The running sum of the Thue-Morse sequence is known to be $2$-synchronised. The automaton representing the running sum, which we call $\bm{\mbox{tmsum}}$, is given at  figure~\ref{tmsum}.  The proof that this is the correct automaton uses an inductive argument. Firstly, the automaton gives the correct answer for $n = 0, 1$. Next, for each $n$, we can show that there is a unique $s$ such that the automaton accepts the pair $(n,s)$. Finally, we can show that if the automaton accepts the pair $(n,s)$, then it also accepts the pair $(n+1, s + \bm{\mbox{T}}(n+1) )$ where $\bm{\mbox{T}}$ is the Thue-Morse sequence. These claims are established by the following  {\tt Walnut} commands.

\bigskip

\begin{figure}[htbp]
   \begin{center}
    \includegraphics[width=6in]{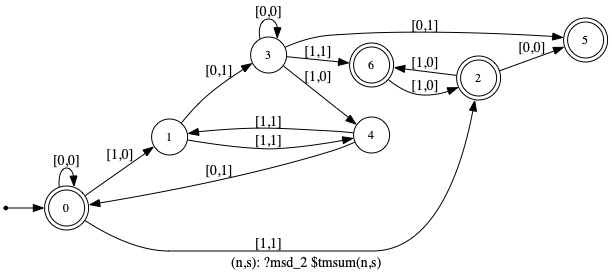}
    \end{center}
    \caption{Automaton $\bm{\mbox{tmsum}}$ for the running sum of the Thue-Morse sequence.}
    \label{tmsum}
\end{figure}

\bigskip

\begin{verbatim}
eval testtmsum "?msd_2 An,s,t $tmsum(n,s) & $tmsum(n,t) => s=t":
eval testtmsum "?msd_2 An Es $tmsum(n,s)":
eval testtmsum "?msd_2 An,s,u ($tmsum(n,s) & T[n+1]=u) 
        => $tmsum(n+1, s+u)":
\end{verbatim}

{\tt Walnut} returns the value TRUE for each of the four statements. The first and second commands show that for each $n$, there is a unique $s$ such that $\bm{\mbox{tmsum}}$ accepts the pair $(n,s)$. The third command shows that $\bm{\mbox{tmsum}}$ does the right thing as $n$ increases.

Next is the Cantor sequence, which is the characteristic sequence of the Cantor integers. An integer is a Canton integer if its base $3$ representation only contains the integers $0$ and $2$. The running sum of the Cantor sequence is $(3, 2)$-synchronised and a synchronising automaton, called $\bm{\mbox{cansum}}$,  is given at figure~\ref{casum}. The correctness of the automaton can again be established using  {\tt Walnut}. The commands are given below. $\bm{\mbox{CA}}$ is  {\tt Walnut}'s name for the Cantor sequence.

\bigskip

\begin{figure}[htbp]
   \begin{center}
    \includegraphics[width=3in]{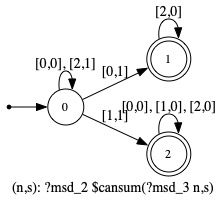}
    \end{center}
    \caption{Automaton $\bm{\mbox{cansum}}$ for the running sum of the Cantor sequence.}
    \label{casum}
\end{figure}

\bigskip

\begin{verbatim}
eval testcs1 "An Es $cansum(?msd_3 n,s)":
eval testcs2 "An,s,t $cansum(?msd_3 n,s) & $cansum(?msd_3 n,t) => s=t":
eval testcs3 "?msd_2 An,s,u ($cansum(?msd_3 n,s) & CA[?msd_3 n+1]=u) 
        => $cansum(?msd_3 n+1, s+u)":
\end{verbatim}

\bigskip

The second-bit sequence $\bm{\mbox{sb}}$ is defined by:
$$
\bm{\mbox{sb}}(n) = a_2 \text{    where   } n = a_1 a_2 \dots  \text{  is the base 2 representation of n in msd form.}
$$

The running sum of the second-bit sequence is $2$-synchronised and the synchronising automaton, called $\bm{\mbox{sbsum}}$, appears at figure~\ref{sbsum}. The correctness of the automaton can be established using  {\tt Walnut} in the same way as for the earlier sequences. 

\bigskip

\begin{figure}[htbp]
   \begin{center}
    \includegraphics[width=3in]{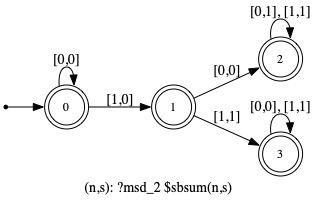}
    \end{center}
    \caption{Automaton $\bm{\mbox{sbsum}}$ for the running sum of the second-bit sequence.}
    \label{sbsum}
\end{figure}

\bigskip

The next sequence in this section is the twisted Thue-Morse sequence,  $\bm{\mbox{ttm}}$. For an integer $n \geq 0$, $\bm{\mbox{ttm}} (n)$ is the number of $0$'s, modulo $2$, in the base $2$ representation of $n$, with $\bm{\mbox{ttm}}(0)$  defined to be $0$. The running sum of the twisted Thue-Morse sequence is $2$-synchronised and the synchronising automaton, called $\bm{\mbox{ttmsum}}$, appears at figure~\ref{ttmsum}. Again,  {\tt Walnut} will verify the correctness of the automaton.

\bigskip

\begin{figure}[htbp]
   \begin{center}
    \includegraphics[width=6in]{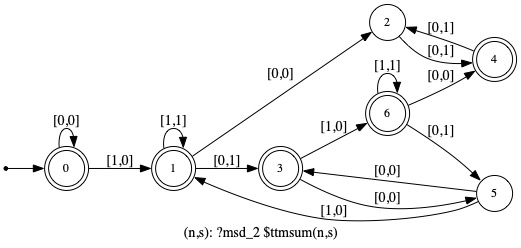}
    \end{center}
    \caption{Automaton $\bm{\mbox{ttmsum}}$ for the running sum of the twisted Thue-Morse sequence.}
    \label{ttmsum}
\end{figure}

\bigskip

Finally, we examine the sequence $\gamma$, the $n$'th value of which is the highest power of $2$ dividing $n!$. The sequence is of interest when considering the values of $n$ for which $n!$ can be written as a sum of three squares.\cite{Burns:2021aa}\cite{Burns:2022aa} The sequence itself is not synchronised because, according to Legendre's formula, $\gamma_n = n - \sum_{k \geq 0} a_k$ when $n$ has binary representation $\sum_{k \geq 0} a_k 2^k$. Hence, if $\gamma$ were synchronised, the sequence which maps $n$ to the sum of its binary digits would also be synchronised by theorem \ref{shallit}. However, that sequence is bounded by  $\mathcal{O} (\log n) = \smallO (n)$, contradicting theorem \ref{shallit2}. Instead, we consider the sequence $\bar{\gamma} = \gamma \pmod{2}$, which is a $2$-automatic sequence. The running sum of $\bar{\gamma}$ is $2$-synchronised. An automaton for $\bar{\gamma}$ is shown in figure  \ref{gamaut} and a synchronising automaton for $\mathbf{sum_{\bar{\gamma}}}$ is shown in figure \ref{gamsumaut}. Both automata take integers in least significant digit first order.

\begin{figure}[htbp]
   \begin{center}
    \includegraphics[width=3in]{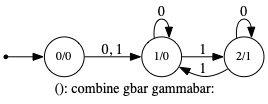}
    \end{center}
    \caption{Automaton for the sequence $\bar{\gamma}$.}
    \label{gamaut}
\end{figure}

\bigskip

\begin{figure}[htbp]
   \begin{center}
    \includegraphics[width=6in]{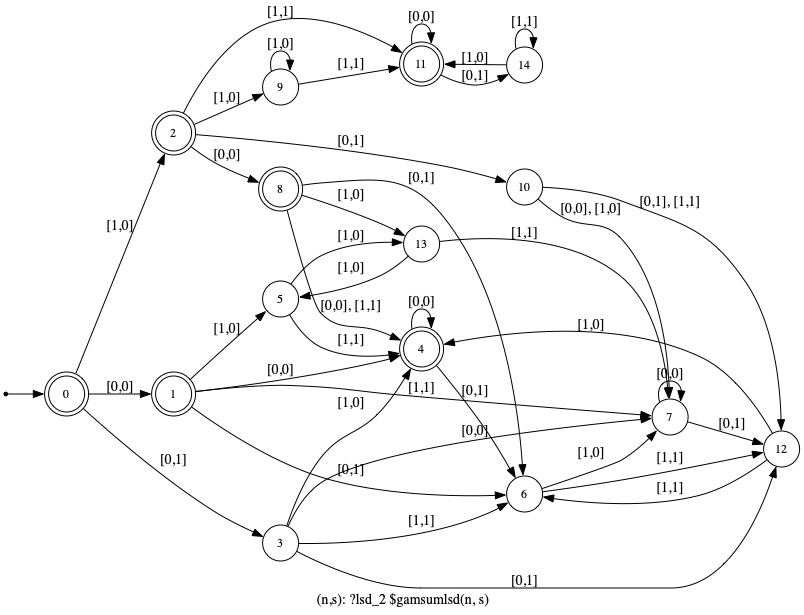}
    \end{center}
    \caption{Synchronising automaton for the running sum of the sequence $\bar{\gamma}$.}
    \label{gamsumaut}
\end{figure}

\bigskip

\section{Unsynchronised sequences}
\label{unsync}

This section examines automatic sequences for which the running sum is not synchronised. We rely heavily on linear representations for the running sums. A linear representation of a function $f$ defined over an alphabet $\{0, 1, \dots, b-1 \}$ consists of vectors $v$ and $w$ and matrices $M_0, M_1, \dots, M_{b-1}$ such that, if $n_k n_{k-1} \dots n_0$  is the base $b$ representation of n in msd form, then $f(n) = v*M_{n_k} * M_{n_{k-1}} * \dots * M_{n_0} *w$.  Given an automatic sequence, {\tt Walnut} will provide the linear representation for the running sum of the sequence. For example, the {\tt Walnut}  command

\begin{verbatim}
eval tmsumrep n "?msd_2 (j>=0) & (j<=n) & T[j]=@1":
\end{verbatim}

produces a linear representation for the running sum of the Thue-Morse sequence. In this case $v$, $w$, $M_0$ and $M_1$ are given by:

\begin{align}
v  = \begin{pmatrix} 1,0,0,0 \end{pmatrix}
\, \, \text{ and } \, \, 
w = \begin{pmatrix} 0,0,1,1 \end{pmatrix}
\end{align}
\begin{align}
M_0  = \begin{pmatrix}  1,0,0,0\\ 0,1,0,1\\ 0,0,1,0\\ 0,1,0,1 \end{pmatrix}
\, \, \text{ and } \, \, 
M_1 = \begin{pmatrix}  0,1,1,0\\ 0,1,0,1\\ 1,0,0,1\\ 0,1,0,1 \end{pmatrix}.
\end{align}

For most sequences in this section, the linear representation of the running sum will be used to calculate the growth rate of a sub-sequence determined by a regular expression. We will then use theorem \ref{shallit} to construct a new sequence with growth rate that conflicts with theorem \ref{shallit2}. As a result, we deduce that the original sequence cannot have been synchronised. A suitable subsequence can be found by looking for a product of the matrices $\{ M_0, \dots, M_{b-1} \}$ for which the minimal polynomial has a repeated non-zero root.

\bigskip

\subsection{Period-doubling sequence}
\label{pdsec}
The period-doubling sequence $\bm{\mbox{pd }} = d_0 d_1 d_2 \dots$ satisfies the relations 
$$
d_{2n} = 1, \, \, \, d_{4n+1} = 0, \, \, \, d_{4n+3} = d_n .
$$
In order to show that $\bf sum_{pd}$, the running sum of the period-doubling sequence,  is not 2-synchronised, we will first calculate the value of $\mathbf{sum_{pd}}$$(n)$ when $n =  (4^{r+1} - 1)/3$ for $r \in  \mathbb{N}$. The base 2 representation of $(4^{r+1} - 1)/3$ is $0* (10)^r 1$.

\bigskip

\begin{lemma}
$\mathbf{sum_{pd}}$$( (4^{r+1} - 1)/3)= (2 * 4^{r+1} + 3r + 1)/9$.
\end{lemma}
\begin{proof}
We can use the linear representation of $\mathbf{sum_{pd}}$ provided by {\tt Walnut}, and the base 2 representation of $(4^{r+1} - 1)/3$, to calculate \mbox{$\mathbf{sum_{pd}}$$(4^{r+1}-1)/3)$}. The linear representation is determined by the vectors $v$ and $w$ and matrices $M_0$ and $M_1$ below:
\begin{align}
v  = \begin{pmatrix} 1,0,0,0 \end{pmatrix}
\, \, \text{ and } \, \, 
w = \begin{pmatrix} 1,1,0,0 \end{pmatrix}
\end{align}
\begin{align}
M_0  = \begin{pmatrix} 1,0,0,0\\ 0,1,0,1\\ 1,0,0,0\\ 0,2,0,0 \end{pmatrix}
\, \, \text{ and } \, \, 
M_1 = \begin{pmatrix} 0,1,1,0\\ 0,1,0,1\\ 1,1,0,0\\ 0,2,0,0 \end{pmatrix}.
\end{align}

\bigskip

The matrix $M_1 * M_0$ has minimal polynomial 
$$
x^4 - 6x^3 + 9x^2 - 4x = (x - 4)  x  (x - 1)^2.
$$ 

The theory of linear recurrences says there are constants $a$, $b$ and $c$ such that

$$
\mathbf{sum_{pd} }((4^{r+1} - 1)/3) = v* (M_1 * M_0)^r * M1 * w = a + br + c4^r.
$$

We then use specific values of $\bf sum_{pd}$ $(n)$ to deduce that $a = 1/9$, $b = 3/9$ and $c = 8/9$ as required. 
\end{proof}

\bigskip

\begin{corollary}
\label{pdsumcor}
The running sum of the period-doubling sequence is not 2-synchronised.
\end{corollary}
\begin{proof}
The function $ \lfloor 2n/3 \rfloor$ is synchronised by theorem \ref{shallit} We define the function $g$ by
$$
g(n) =
\begin{cases}
  \mathbf{sum_{pd} }(n) \dotdiv \lfloor 2n/3 \rfloor, & \text{if } \, \, n =  (4^{r+1} - 1)/3 \\
    0 & \text{otherwise}
\end{cases}
$$
If $\mathbf{sum_{pd}}$ $(n)$ is synchronised, then so is $\bf sum_{pd}$ $(n) \dotdiv \lfloor (2n+1)/3 \rfloor$ by theorem~\ref{shallit}. Hence, $g$ is also synchronised since the set $\{ (4^{r+1} - 1)/3: r \in \mathbb{N} \}$ comes from a regular expression. When $n =  (4^{r+1} - 1)/3$,
\begin{align*}
g(n) &= (2 * 4^{r+1} + 3r + 1)/9 \dotdiv \lfloor (2 * 4^{r+1}-2)/9 \rfloor \\
                                          & = \lfloor(r+1)/3 \rfloor.
\end{align*}

\bigskip

Since $g(n) = \mathcal{O}(\log n) = \smallO(n)$ and is unbounded, it cannot be a synchronised function by theorem~\ref{shallit2}. Hence, $\bf sum_{pd}$ is not $2$-synchronised.
\end{proof}

\bigskip

\subsection{Mephisto Waltz sequence}
\label{mwsec}

The Mephisto Waltz sequence $\bm{\mbox{mw }} = m_0, m_1, \dots$ is a $3$-automatic sequence in which the value of $m_n$ is given by the number of $2$'s modulo $2$ in the base $3$ representation of $n$. The running sum of $\bm{\mbox{mw}}$, $\mathbf{sum_{mw}}$ , has a linear representation given by the vectors  $v$ and $w$ and matrices $M_0, M_1$ and $M_2$ below:

\bigskip

\begin{align}
v  = \begin{pmatrix} 1,0,0,0 \end{pmatrix}
\, \, \text{ and } \, \, 
w = \begin{pmatrix} 0,0,1,1 \end{pmatrix}
\end{align}
\begin{align}
M_0  = \begin{pmatrix}  1,0,0,0\\ 0,2,0,1\\ 0,0,1,0\\ 0,1,0,2 \end{pmatrix}
\, \, \, \, 
M_1  = \begin{pmatrix}  1,1,0,0\\ 0,2,0,1\\ 0,0,1,1\\ 0,1,0,2 \end{pmatrix}
\, \, \text{ and } \, \, 
M_2 = \begin{pmatrix}  0,2,1,0\\ 0,2,0,1\\ 1,0,0,2\\ 0,1,0,2 \end{pmatrix}
\end{align}

\bigskip

\begin{lemma}
When $n =  (3^{2r} - 1)/2$, $\mathbf{sum_{mw}}$$(n) = n/2 - r$.
\end{lemma}
\begin{proof}
The matrix $M_1^2$ has minimal polynomial 
$$
x^3 - 11x^2 + 19x - 9 = (x - 9)  (x - 1)^2
$$ 

Hence, there are constants $a$, $b$ and $c$ such that

$$
\mathbf{sum_{mw}} ( (3^{2r} - 1)/2) = v* (M_1 )^{2r} * w = a + br + c3^{2r}.
$$

We then use specific values of $\mathbf{sum_{mw}}$$(n)$ to deduce that $a = -1/4$, $b = -1$ and $c = 1/4$. Therefore, when $n =   (3^{2r} - 1)/2$, 

\begin{align*}
\mathbf{sum_{mw}}(n) &= -1/4 - r + 3^{2r}/4 =(3^{2r} - 1)/4 - r\\
                                                      &= n/2 - r
\end{align*}                                               
\end{proof}

\bigskip

\begin{corollary}
\label{mwsumcor}
The running sum of the Mephisto Waltz sequence is not $3$-synchronised.
\end{corollary}
\begin{proof}
We consider the function

$$
g(n) =
\begin{cases}
    n/2  \dotdiv \mathbf{sum_{mw}}(n), & \text{if } \, \, n =   (3^{2r} - 1)/2 \\
    0 & \text{otherwise.}
\end{cases}
$$

Then $g(n) = \mathcal{O}(\log n) = \smallO(n)$ and is unbounded, and so cannot be a synchronised function by theorem~\ref{shallit2}. Hence, $\mathbf{sum_{mw}}$ is not $3$-synchronised.
\end{proof}

\bigskip

\subsection{Paper-folding sequence}
\label{pfsec}

The regular paper-folding sequence $\bm{\mbox{pf }} = f_0 f_1 f_2 \dots$ satisfies the relations
$$
f_{2n+1} = f_{n}, \, \, \, f_{4n} = 0, \, \, \, f_{4n+2} = 1.
$$
The argument to show that $\mathbf{sum_{pf}}$,  the running sum of the paper-folding sequence, is not $2$-synchronised, is similar to that used for the period-doubling sequence. We examine $\mathbf{sum_{pf}}$  when $n$ has the binary representation given by the regular expression $(1 0 )^r 1$. Such an $n = (4^{r+1}-1)/3$. A linear representation for $\mathbf{sum_{pf}}$ is given by the vectors  $v$ and $w$ and matrices $M_0$ and $M_1$ below:

\bigskip

\begin{align}
v  = \begin{pmatrix} 1, 0, 0, 0, 0, 0, 0, 0 \end{pmatrix}
\, \, \text{ and } \, \, 
w = \begin{pmatrix} 0, 0, 0, 0, 1, 1, 1, 1 \end{pmatrix}
\end{align}
\begin{align}
M_0  = \begin{pmatrix}  
1,0,0,0,0,0,0,0\\ 0,1,0,1,0,0,0,0\\ 0,0,0,0,1,0,0,0\\ 0,0,0,1,0,1,0,0\\ 1,0,0,0,0,0,0,0\\ 0,1,0,0,0,0,0,1\\ 0,0,0,0,1,0,0,0\\ 0,0,0,0,0,1,0,1 \end{pmatrix}
\, \, \text{ and } \, \, 
M_1 = \begin{pmatrix}  
0,1,1,0,0,0,0,0\\ 0,1,0,1,0,0,0,0\\ 0,0,1,0,0,1,0,0\\ 0,0,0,1,0,1,0,0\\ 0,1,0,0,0,0,1,0\\ 0,1,0,0,0,0,0,1\\ 0,0,0,0,0,1,1,0\\ 0,0,0,0,0,1,0,1 \end{pmatrix}
\end{align}

\bigskip

\begin{lemma}
When $n =  (4^{r+1}-1)/3$, $\mathbf{sum_{pf}}$$(n) =  (n+1)/2 - r$.
\end{lemma}
\begin{proof}
The matrix $M_1 * M_0$ has minimal polynomial 
$$
x^4 - 6x^3 + 9x^2 - 4x = (x - 4)  x  (x - 1)^2.
$$ 

There are constants $a$, $b$ and $c$ such that

$$
\mathbf{sum_{pf}} ((4^{r+1} - 1)/3) = v* (M_1 * M_0)^r * M1 * w = a + br + c4^r.
$$

We then use specific values of $\mathbf{sum_{pf}}$$(n)$ to deduce that $a = 1/3$, $b = -1$ and $c = 2/3$. Therefore, when $n =  (4^{r+1}-1)/3$, 

\begin{align*}
\mathbf{sum_{pf}}(n) &= 1/3 - r + 2*4^r/3 = (2*4^r + 1)/3 - r\\
                                                      &= (n+1)/2 - r
\end{align*}                                               
\end{proof}

\bigskip

\begin{corollary}
\label{pfsumcor}
The running sum of the paper-folding sequence is not 2-synchronised.
\end{corollary}
\begin{proof}
We use the same argument as corollary \ref{pdsumcor}. This time we consider the function

$$
g(n) =
\begin{cases}
    (n+1)/2  \dotdiv \mathbf{sum_{pf}}(n), & \text{if } \, \, n =  (4^{r+1} - 1)/3 \\
    0 & \text{otherwise.}
\end{cases}
$$

Again, $g(n) = \mathcal{O}(\log n) = \smallO(n)$ and is unbounded, and so cannot be a synchronised function by theorem~\ref{shallit2}. Hence, $\mathbf{sum_{pf}}$ is not $2$-synchronised.
\end{proof}

\bigskip

\subsection{Leech's word}
\label{lwsec}
Leech's word $\bm{\mbox{le}}$ is a 13-automatic sequence defined over the alphabet $\{0, 1, 2 \}$. It is the fixed point of the 13-uniform morphism, starting from $0$:
\begin{align*}
0 \rightarrow \,\, &0,1,2,1,0,2,1,2,0,1,2,1,0\\
1 \rightarrow \,\, &1,2,0,2,1,0,2,0,1,2,0,2,1\\
2 \rightarrow \,\, &2,0,1,0,2,1,0,1,2,0,1,0,2
\end{align*}
and appears as sequence A337005 in the OEIS.\cite{oeis} The running sum of $\bm{\mbox{le}}$, $\mathbf{sum_{le}}$, is given by
$$
\mathbf{sum_{le}} = \mathbf{sum_{le1}} + 2 * \mathbf{sum_{le2}}
$$
where $\mathbf{sum_{le1}}$ is the running sum of the number of $1$'s appearing in $\mathbf{le}$ and $\mathbf{sum_{le2}}$ is the running sum of the number of $2$'s appearing in $\mathbf{le}$. We can obtain linear representations for $\mathbf{sum_{le1}}$ and $\mathbf{sum_{le2}}$ with the {\tt Walnut} commands

\begin{verbatim}
eval lesum1 n "?msd_13  (j>=0) & (j<=n) & LE[j]=@1":
eval lesum2 n "?msd_13  (j>=0) & (j<=n) & LE[j]=@2":
\end{verbatim}

$\bm{\mbox{LE}}$ is  {\tt Walnut}'s name for Leech's word. Each of these commands produces vectors $v$ and $w$ and matrices $M_0, M_1, \dots , M_{12}$. In each case the vectors  $v$ and the matrices are identical. The vectors $w$ differ and will be denoted by $w_1$ and $w_2$. We will calculate $\mathbf{sum_{le}}$$(n)$ when $n$ has the base 13 representation $(111)^r$, and so we will only need the matrix $M_1$. $v$, $w_1$, $w_2$ and $M_1$ are given below:

\begin{align*}
v  = \begin{pmatrix} 1, 0, 0, 0, 0, 0 \end{pmatrix}
\, \, \text{, } \, \, 
w_1 = \begin{pmatrix} 0, 0, 1, 1, 0, 0 \end{pmatrix}
\, \, \text{ and } \, \, 
w_2 = \begin{pmatrix} 0, 0, 0, 0, 1, 1 \end{pmatrix}
\end{align*}
\begin{align*}
M_1  = \begin{pmatrix} 0,1,1,0,0,0\\  0,4,0,5,0,4\\ 0,0,0,1,1,0\\ 0,4,0,4,0,5\\ 1,0,0,0,0,1\\ 0,5,0,4,0,4 \end{pmatrix}
\end{align*}

\bigskip

A linear representation for $\mathbf{sum_{le}}$ is determined by the matrices $M_0, M_1, \dots, M_{12}$ and the vectors $v$ and $w_1 + 2w_2$.

\bigskip

\begin{lemma}
\label{lelem}
For  $n = (13^{3r} - 1)/12$, where $r \in \mathbb{N}$, $\mathbf{sum_{le}}$$(n) = n + 3r$.
\end{lemma}
\begin{proof}
The matrix $M_1$ has minimal polynomial 
$$
x^6 - 12x^5 - 12x^4 - 14x^3 + 12x^2 + 12x + 13 = (x - 13)  (x - 1)  (x - \theta)^2  (x + \theta + 1)^2
$$ 
where $\theta = (-1 + \sqrt{-3})/2$.
There are constants $a$, $b$ and $c$, $d$, $e$ and $f$ such that
$$
\mathbf{sum_{le}} ((13^r - 1)/12) = v* M_1 ^r *(w_1 + 2*w_2) = a + b \, 13^r + (c + dr)\theta^r + (e+fr)(-1-\theta)^r.
$$

We use the specific values 
\begin{align*}
\mathbf{sum_{le}}((13^1 - 1)/12) &= 1, \,\,\, \mathbf{sum_{le}}((13^2 - 1)/12) = 16\\ 
\mathbf{sum_{le}}((13^3 - 1)/12) &= 186, \,\,\, \mathbf{sum_{le}}((13^4 - 1)/12) = 2377\\ 
\mathbf{sum_{le}}((13^5 - 1)/12) &= 30943, \,\,\, \mathbf{sum_{le}}((13^6 - 1)/12) = 402240
\end{align*}
to deduce that
$$
a = 11/12, b = 1/12, c = (\theta - 1)/3, d = (\theta + 2)/3, e =  (-\theta - 2)/3, f =  (-\theta + 1)/3.
$$ 

Therefore, when $n =  (13^{3r}-1)/12$, 
\begin{align*}
\mathbf{sum_{le}}(n) =& (13^{3r}+11)/12 + (\theta - 1 + 3r\theta + 6r)\theta^{3r}/3 + (-\theta - 2 - 3r\theta + 3r)(-1-\theta)^{3r}/3\\
=& (13^{3r}+11)/12 + (\theta - 1 + 3r\theta + 6r -\theta - 2 - 3r\theta + 3r)/3\\
=& (13^{3r} - 1)/12 + 3r\\ 
=& n + 3r
\end{align*}
where we have used the fact that $\theta^3 = (-1-\theta)^3 = 1$.
\end{proof}

\bigskip

\begin{corollary}
The running sum of Leech's word is not $13$-synchronised.
\end{corollary}
\begin{proof}
We define the function $g$ by
$$
g(n) =
\begin{cases}
    \mathbf{sum_{le}}(n) \dotdiv n , & \text{if } \, \, n  \text{ has base 13 representation } \,\, (111)^* \\
    0, & \text{otherwise}
\end{cases}
$$
If $\mathbf{sum_{le}}$ is synchronised, then so is $g$ by theorem~\ref{shallit}. But from lemma \ref{lelem}, $g = \mathcal{O}(\log n) = \smallO(n)$ and is unbounded, and so it cannot be a synchronised function by theorem~\ref{shallit2}. Hence, $\mathbf{sum_{le}}$ is not $13$-synchronised.
\end{proof}

\bigskip

\subsection{Baum-Sweet sequence}

The Baum-Sweet sequence,  $\bm{\mbox{bs}}$, satisfies
$$
\bm{\mbox{bs}} (n) =
\begin{cases}
0, & \text{if the base 2 representation of } \, \, n \,\, \text{has a block of zeros of odd length} \\
1, & \text{otherwise }
\end{cases}
$$
The running sum of $\bm{\mbox{bs}}$, $\mathbf{sum_{bs}}$ , has a linear representation given by the vectors  $v$ and $w$ and matrices $M_0$ and $M_1$ below:

\bigskip

\begin{align}
v  = \begin{pmatrix} 1, 0, 0, 0, 0, 0 \end{pmatrix}
\, \, \text{ and } \, \, 
w = \begin{pmatrix} 1, 1, 1, 1, 0, 0 \end{pmatrix}
\end{align}
\begin{align}
M_0  = \begin{pmatrix} 
1,0,0,0,0,0\\  0,1,0,1,0,0\\ 0,0,0,0,1,0\\ 0,0,0,1,0,1\\ 0,0,1,0,0,0\\ 0,0,0,1,0,0 \end{pmatrix}
\, \, \text{ and } \, \, 
M_1 = \begin{pmatrix}  
0,1,1,0,0,0\\ 0,1,0,1,0,0\\ 0,0,1,0,0,1\\ 0,0,0,1,0,1\\ 0,0,0,1,0,0\\ 0,0,0,1,0,0 \end{pmatrix}
\end{align}

\bigskip

$M_0$ has minimal polynomial $x^4 - x^3 - 2x^2 + x + 1 =  (x - \phi)  (x - 1)  (x - \psi)  (x + 1)$ where $\phi$ is the golden ration and $\psi$ is its conjugate.

\bigskip

\begin{theorem}
\label{bslem}
The running sum of $\bf bs$ is not 2-synchronised.
\end{theorem}
\begin{proof}
The linear representation above shows that, for $k \in \mathbb{N}$, there are constants  $a, b, c, d$ such that $\mathbf{sum_{bs}}$$(2^k) = a + b (-1)^k +  c \phi^k  + d \psi^k$. The values of $(a, b, c, d)$ can be calculated, giving 
$$
\mathbf{sum_{bs}}(2^k) = 1/2 + (-1)^k/2  +  (3\sqrt{5}/10 + 1/2) \phi^k +  (-3\sqrt{5}/10 + 1/2) \psi^k .
$$

Since $\mathbf{sum_{bs}}$ is unbounded and $\mathbf{sum_{bs}}$$(2^k) \sim \phi^k  = \smallO(2^k)$, $\mathbf{sum_{bs}}$ cannot be a 2-synchronised function by theorem~\ref{shallit2}.
\end{proof}

\bigskip

\subsection{Stewart's choral sequence}

\bigskip

Stewart choral sequence, $\bm{\mbox{sc}} = s_0 s_1 s_2 \dots$, is a $3$-automatic sequence satisfying the recurrence relation
$$
s_{3n} = 0, \,\,\, s_{3n-1} = 1, \,\,\, s_{3n+1} = s_{n}.
$$
The running sum of $\bm{\mbox{sc}}$, $\mathbf{sum_{sc}}$, has a linear representation determined by vectors $v$ and $w$ and matrices $M_0$, $M_1$ and $M_2$.  $v$, $w$ and $M_1$ are given below:
\begin{align}
v  = \begin{pmatrix} 1,0,0,0 \end{pmatrix}
\, \, \text{ and } \, \, 
w = \begin{pmatrix} 0,0,1,1 \end{pmatrix}
\end{align}
\begin{align}
M_1 = \begin{pmatrix}  1,1,0,0\\ 0,2,0,1\\ 0,1,1,0\\ 0,1,0,2 \end{pmatrix}.
\end{align}

\bigskip

\begin{lemma}
\label{sclem}
Let $r \in  \mathbb{N}$. When $n = (3^{r} - 1)/2$, $\mathbf{sum_{sc}}$$(n) = (3^r - 2r - 1)/4$. 
\end{lemma}
\begin{proof}
The matrix $M_1$ has minimal polynomial 
$$
x^3 - 5x^2 + 7x - 3 = (x - 3) (x - 1)^2.
$$ 

There are constants $a$, $b$ and $c$ such that

$$
\mathbf{sum_{sc}}((3^{r} - 1)/2) = v* M_1^r * w = a + br + c3^r.
$$

We then use specific values of $\mathbf{sum_{sc}}$$(n)$ to deduce that $a = -1/4$, $b = -1/2$ and $c = 1/4$. Therefore,

$$
\mathbf{sum_{sc}} ((3^{r} - 1)/2) = -1/4 - r/2 + 3^r/4 =  (3^r - 2r - 1)/4.
$$                                         
\end{proof}

\bigskip

\begin{corollary}
The running sum of Stewart's choral sequence is not 3-synchronised. 
\end{corollary}
\begin{proof}
Use the same argument as for corollaries \ref{pfsumcor}, \ref{pdsumcor} and theorem \ref{bslem}.
\end{proof}

\bigskip

\subsection{Rudin-Shapiro sequence}
\label{rs}

The Rudin-Shapiro sequence,  $\bm{\mbox{rs}}$, is a $2$-automatic sequence satisfying
\begin{align*}
\bm{\mbox{rs}} &(0) = \bm{\mbox{rs}} (1) = 0\\
\bm{\mbox{rs}} &(2n)  = \bm{\mbox{rs}} (n)  \\
\bm{\mbox{rs}} &(4n + 1) = \bm{\mbox{rs}} (n) \\
\bm{\mbox{rs}} &(4n + 3)  = 1 - \bm{\mbox{rs}} (2n+1).
\end{align*}

\bigskip

\begin{theorem}
The running sum of $\bm{\mbox{rs}}$, $\mathbf{sum_{rs}}$, is not $2$-synchronised. 
\end{theorem}
\begin{proof}
{\tt Walnut} provides the following linear representation for $\mathbf{sum_{rs}}$:

\bigskip

\begin{align}
v  = \begin{pmatrix} 1,0,0,0,0,0,0,0 \end{pmatrix}
\, \, \text{ and } \, \, 
w = \begin{pmatrix} 0,0,0,0,1,1,1,1 \end{pmatrix}
\end{align}
\begin{align}
M_0  = \begin{pmatrix}  1,0,0,0,0,0,0,0\\ 0,1,0,1,0,0,0,0\\ 1,0,0,0,0,0,0,0\\ 0,1,0,0,0,1,0,0\\ 0,0,0,0,0,0,1,0\\ 0,0,0,1,0,0,0,1\\ 0,0,0,0,0,0,1,0\\ 0,0,0,0,0,1,0,1 \end{pmatrix}
\, \, \text{ and } \, \, 
M_1 = \begin{pmatrix}  
0,1,1,0,0,0,0,0\\ 0,1,0,1,0,0,0,0\\ 0,1,0,0,1,0,0,0\\ 0,1,0,0,0,1,0,0\\ 0,0,1,0,0,0,0,1\\ 0,0,0,1,0,0,0,1\\ 0,0,0,0,1,0,0,1\\ 0,0,0,0,0,1,0,1 \end{pmatrix}
\end{align}

\bigskip

The minimal polynomial for $M_1$ is 
$$
x^6 - 2x^5 - 3x^4 + 6x^3 + 2x^2 - 4x = (x - 2)  (x - 1)  x  (x + 1)  (x^2 - 2).
$$ 
Therefore, there are constants $a, b, c, d, e$ such that
$$
\mathbf{sum_{rs}} (2^k - 1) =  v*M_1^k*w = a + b (-1)^k +  c 2^k  + d (\sqrt{2})^k + e (-\sqrt{2})^k.
$$

The constants $(a, b, c, d, e) = (0, 0, 1/2, -(\sqrt{2} + 1)/4, (\sqrt{2} - 1)/4)$, so
\begin{equation}
\label{rseq}
\mathbf{sum_{rs}}(2^k - 1) = 
\begin{cases}
2^{k-1} - 2^{k/2 - 1}, & \text{if k is even} \\
2^{k-1} - 2^{(k-1)/2}, & \text{if k is odd.}
\end{cases}
\end{equation}

\bigskip

From equation (\ref{rseq}), the base $2$ representation of $\mathbf{sum_{rs}}$$(2^k - 1)$ is given by $1^{k/2} 0^{k/2 - 1}$ when $k$ is even and by $1^{(k-1)/2} 0^{(k-1)/2}$ when $k$ is odd. The pumping lemma then says that an automaton which accepts pairs $(2^k - 1, \mathbf{sum_{rs}}(2^k - 1))$, for $k$ large enough, must also accept base $2$ formatted pairs $(1^{k+m}, 1^{m + k/2} 0^{k/2 - 1})$ or $(1^{k+m}, 1^{m + (k-1)/2} 0^{(k-1)/2})$ where m is some non-zero integer. This contradicts the value of $\mathbf{sum_{rs}}$$(2^{k+m} - 1)$ given in equation (\ref{rseq}) and so the running sum of the Rudin-Shapiro sequence is not $2$-synchronised.
\end{proof}

\bigskip

\subsection{Fibonacci-Thue-Morse sequence}	
\label{ftm}

The Fibonacci-Thue-Morse sequence,  $\bm{\mbox{ftm}}$, satisfies
$$
\bm{\mbox{ftm}} (n) =
\begin{cases}
0 & \text{if the Fibonacci representation of } \, \, n \,\, \text{has an even number of 1's} \\
1 & \text{if the Fibonacci representation of } \, \, n \,\, \text{has an odd number of 1's.}
\end{cases}
$$
In order to show that the running sum of $\bm{\mbox{ftm}}$, $\mathbf{sum_{ftm}}$, is not Fibonacci-synchronised, we look at the values of $\mathbf{sum_{ftm}}$$(n)$ when $n = n(r)$ has the Fibonacci representation  $(100 100)^r = (100)^{2r}$. {\tt Walnut} provides the following linear representation for $\mathbf{sum_{ftm}}$:

\bigskip

\begin{align}
v  = \begin{pmatrix} 1, 0, 0, 0, 0, 0, 0, 0, 0, 0, 0, 0 \end{pmatrix}
\, \, \text{ and } \, \, 
w = \begin{pmatrix} 0, 0, 1, 0, 1, 1, 1, 1, 1, 0, 0, 0 \end{pmatrix}
\end{align}
\begin{align}
M_0  = \begin{pmatrix}  1,0,0,0,0,0,0,0,0,0,0,0\\  0,0,0,1,1,0,0,0,0,0,0,0\\ 0,0,0,0,0,1,0,0,0,0,0,0\\ 0,0,0,1,1,0,0,0,0,0,0,0\\ 0,0,0,0,0,0,0,1,0,0,0,0\\ 0,0,0,0,0,1,0,0,0,0,0,0\\ 0,0,0,0,0,0,0,1,0,0,0,0\\ 0,0,0,0,0,0,0,1,0,0,1,0\\ 0,0,0,0,0,0,0,1,0,0,1,0\\ 1,0,0,0,0,0,0,0,0,0,0,0\\ 0,0,0,1,0,0,0,0,0,0,0,0\\ 0,0,0,1,0,0,0,0,0,0,0,0 \end{pmatrix}
\, \, \text{ and } \, \, 
M_1 = \begin{pmatrix}  0,1,1,0,0,0,0,0,0,0,0,0\\ 0,0,0,0,0,0,0,0,0,0,0,0\\ 0,0,0,0,0,0,0,0,0,0,0,0\\ 0,1,0,0,0,0,1,0,0,0,0,0\\ 0,0,0,0,0,0,0,0,1,0,0,0\\ 0,0,0,0,0,0,0,0,1,1,0,0\\ 0,0,0,0,0,0,0,0,0,0,0,0\\ 0,0,0,0,0,0,0,0,1,0,0,1\\ 0,0,0,0,0,0,0,0,0,0,0,0\\ 0,0,0,0,0,0,0,0,0,0,0,0\\ 0,1,0,0,0,0,0,0,0,0,0,0\\ 0,0,0,0,0,0,0,0,0,0,0,0 \end{pmatrix}
\end{align}

\bigskip

\begin{lemma}
\label{fibpoly}
\begin{enumerate}[label=(\roman*)]
    \item $F_{6r+14} - 18F_{6r+8} + F_{6r+2} = 0$ for all $r \in  \mathbb{N}$.
    \item $F_{6r+14} - 6F_{6r+8} - 16F_{6r+7} -16F_{6r+4} +5F_{6r+2} = 0$ for all $r \in  \mathbb{N}$.
\end{enumerate}
\end{lemma}
\begin{proof}
The first statement follows from Binet's formula (\ref{binet}) and the fact that the polynomial $x^2 - x - 1$ is a factor of $x^{12} - 18x^6 + 1$. The second statement follows from Binet's formula and the fact that  $x^2 - x - 1$ is a factor of $x^{12} - 6x^6 - 16x^5 - 16x^2 + 5$.
\end{proof}

\bigskip

\begin{lemma}
\label{fib1}
Let $n = n(r)$ have Fibonacci representation $(100 100)^r$. Then, 

$\mathbf{sum_{ftm}}$$(n) = a(r) + b(r)$,  where $a(r) = (F_{6r+2} - 1)/4$ and 

$b(r) = (F_{6r+8} - 13 F_{6r+2} - 32r - 8)/32$.
\end{lemma}
\begin{proof}
Let $A = M_1 * M_0 * M_0$. Then,  

\bigskip

\begin{align}
\label{A2}
A^2  = \begin{pmatrix} 1, 0, 0, 5, 2, 0, 0, 5, 0, 0, 4, 0\\ 0,0,0,0,0,0,0,0,0,0,0,0\\ 0,0,0,0,0,0,0,0,0,0,0,0\\ 0, 0, 0, 7, 4, 0, 0, 6, 0, 0, 4, 0\\ 0, 0, 0, 4, 3, 0, 0, 4, 0, 0, 2, 0\\ 0, 0, 0, 5, 4, 1, 0, 5, 0, 0, 2, 0\\ 0,0,0,0,0,0,0,0,0,0,0,0 \\ 0, 0, 0, 6, 4, 0, 0, 7, 0, 0, 4, 0\\ 0,0,0,0,0,0,0,0,0,0,0,0\\ 0,0,0,0,0,0,0,0,0,0,0,0\\ 0, 0, 0, 4, 2, 0, 0, 4, 0, 0, 3, 0\\ 0,0,0,0,0,0,0,0,0,0,0,0  \end{pmatrix}
\end{align}

\bigskip

Using the linear representation for $\mathbf{sum_{ftm}}$, we prove by induction that
$$
v*(A^{2r}) = (1, 0, 0, a(r), b(r), 0, 0, a(r), 0, 0, b(r) + 2r, 0)
$$
where the functions $a(r)$ and $b(r)$ are given in the statement of the lemma. We have
\begin{align*}
&v*(A^{2r+2})  =  (1, 0, 0, a(r), b(r), 0, 0, a(r), 0, 0, b(r) + 2r, 0)*A^2 \\
&= (1, 0, 0, 5 + 7a(r) + 4b(r) + 6a(r) + 4b(r) + 8r, 2 + 4a(r) + 3b(r) + 4a(r) + 2b(r) + 4r, \\
&0, 0, 5 + 6a(r) + 4b(r) + 7a(r) + 4b(r) + 8r, 0, 0, 4 + 4a(r) + 2b(r) + 4a(r) + 3b(r) + 6r ) \\
&= (1, 0, 0, 5 + 13a(r) + 8b(r) + 8r, 2 + 8a(r) + 5b(r) + 4r,  0, 0, 5 + 13a(r) + 8b(r) + 8r, \\
&0, 0, 4 + 8a(r) + 5b(r) + 6r ) 
\end{align*}
From the definitions of $a(r)$ and $b(r)$, we have:
\begin{align*}
5 + 13a(r)& + 8b(r) + 8r \\
                & = 5 + 13( (F_{6r+2} - 1)/4) +  (F_{6r+8} - 13 F_{6r+2} - 32r - 8)/4 +8r \\
                & = (F_{6r+8}  - 1)/4 \\
                & = a(r+1).
\end{align*}

Again using the definitions of $a(r)$ and $b(r)$, we have:
\begin{align*}
8a(r) + &5b(r) + 4r + 2 \\
            & = (5 F_{6r+8} - F_{6r+2} - 32r - 40)/32 \\
            & = (F_{6r+14} - 13 F_{6r+8} - 32r - 40)/32 \\
            & = b(r+1)
\end{align*}
where the second last step follows from lemma \ref{fibpoly}. This concludes the inductive proof.
\end{proof}

\bigskip

\begin{lemma}
\label{fib2}
Let $n = n(r)$ have Fibonacci representation $(100 100)^r$. Then, 

$\lfloor n/2 \rfloor = a(r) + b(r) + r$, where $a(r) = (F_{6r+2} - 1)/4$ and 

$b(r) = (F_{6r+8} - 13 F_{6r+2} - 32r - 8)/32$.
\end{lemma}
\begin{proof}
We first note that $a(r) + b(r) + r =  (F_{6r+8} - 5F_{6r+2} - 16)/32$.  We use induction to show that, $n(r) = 2a(r) + 2b(r) + 2r$. The inductive step is:
\begin{align*}
n(r+1) &= (100 100)^{r+1} = n(r) + F_{6r+4} + F_{6r+7} \\
           & = 2a(r) + 2b(r) + 2r + F_{6r+4} + F_{6r+7} \\
           & = (F_{6r+8} + 16F_{6r+7} +16F_{6r+4} - 5F_{6r+2} - 16)/16 \\
           & = (F_{6r+8} + F_{6r+14} - 6F_{6r+8} - 16)/16 \\
           & = (F_{6r+14} - 5F_{6r+8} - 16)/16 \\
           & = 2a(r+1) + 2b(r+1) + 2(r+1) 
\end{align*}
where we have used lemma \ref{fibpoly} and the definition of $a(r)$ and $b(r)$. This completes the induction.
\end{proof}

\bigskip

\begin{figure}[htbp]
   \begin{center}
    \includegraphics[width=6in]{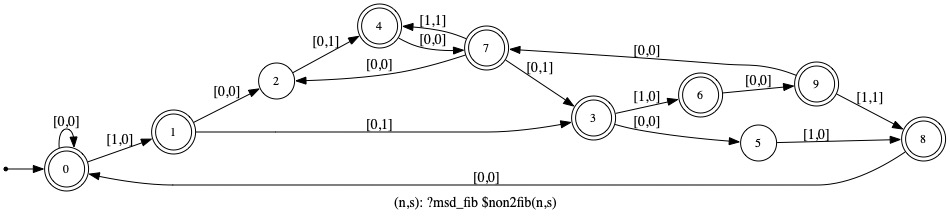}
    \end{center}
    \caption{Fibonacci-synchronising automaton for the function $\lfloor n/2 \rfloor$.}
    \label{non2fib}
\end{figure}

\bigskip

\begin{corollary}
The running sum of the Fibonacci-Thue-Morse sequence is not Fibonacci-synchronised.
\end{corollary}
\begin{proof}
The sequence $\lfloor n/2 \rfloor$ is Fibonacci-synchronised. A synchronising automaton, taken from the paper \cite{Cloitre:2023aa}, is shown at figure~\ref{non2fib}. Define the function $g$ by
$$
g(n) =
\begin{cases}
      \lfloor n/2 \rfloor \dotdiv \mathbf{sum_{ftm}}(n), & \text{if } \, \, n \,\,  \text{has Fibonacci representation } \,\, (100100)^r\\
    0 & \text{otherwise}
\end{cases}
$$
If $\mathbf{sum_{ftm}}$$(n)$ is Fibonacci-synchronised, then so is $ \lfloor n/2 \rfloor \dotdiv \mathbf{sum_{ftm}}(n)$ by theorem~\ref{shallit}, which also applies to Fibonacci representations. Hence, $g$ is also Fibonacci-synchronised since the set of numbers having Fibonacci representation of the form $(100 100)^r$ is defined by a regular expression. If the Fibonacci-representation of $n$ has the form $(100 100)^r$, then $g(n) = r$ by lemmas \ref{fib1} and \ref{fib2}.  Since $g(n) = \mathcal{O}(\log n) = \smallO(n)$ and is unbounded, it cannot be a Fibonacci-synchronised function by theorem~\ref{shallit2}. Hence, $\mathbf{sum_{ftm}}$ is not Fibonacci-synchronised.
\end{proof}

\bigskip

\bibliographystyle{plain}
\begin{small}
\bibliography{Synch}
\end{small}

\end{document}